\documentclass[11pt]{article}
\usepackage{amsmath, amssymb, theorem, latexsym, epsfig, graphics, eufrak}
\usepackage{subfigure}
\numberwithin{equation}{section}

\theoremstyle{plain}
\theorembodyfont{\itshape}
\newtheorem{theorem}{Theorem}[section]
\newtheorem{proposition}[theorem]{Proposition}
\newtheorem{lemma}[theorem]{Lemma}
\newtheorem{corollary}[theorem]{Corollary}
                                                                                
\theorembodyfont{\rmfamily}
\newtheorem{definition}[theorem]{Definition}

\newtheorem{remark}[theorem]{Remark}

\newenvironment{proof}{{\noindent \textbf{Proof}\,\,}}{\hspace*{\fill}$\Box$\medskip}

\title{On determinants of modified Bessel functions and entire solutions of double confluent Heun equations}
\author{V.M.Buchstaber\thanks{Permanent address: Steklov Mathematical Institute, 8, Gubkina street, 119991, Moscow, Russia.  Email: buchstab@mi.ras.ru}
\thanks{All-Russian Scientific Research Institute for Physical and Radio-Technical Measurements (VNIIFTRI),}
\thanks{Supported by part by RFBR grant 14-01-00506.}, 
A.A.Glutsyuk\thanks{ CNRS, France (UMR 5669 (UMPA, ENS de Lyon) and UMI 2615 (Lab. J.-V.Poncelet)), Lyon, France. 
Email:
aglutsyu@ens-lyon.fr}, 
\thanks{National Research University Higher School of Economics (HSE), Moscow, Russia}
 \thanks{Supported by part by RFBR grants 13-01-00969-a, 16-01-00748, 16-01-00766 
 and ANR grant ANR-13-JS01-0010.}}
\begin{document}
\maketitle
\def\zz{\mathbb Z}
\def\nn{\mathbb N}
\def\var{\varepsilon}
\def\td{\mathbb T^3} 
\def\rr{\mathbb R}
\def\la{\lambda}
\def\go#1{\EuFrak#1}
\def\wt{\widetilde}
\def\sign{\operatorname{sign}}
\def\cc{\mathbb C}
\def\diag{\operatorname{diag}}
\def\dd{\Delta_{discr}}

\vspace{-1cm}
\begin{abstract} We investigate the question on existence  of entire solutions of well-known linear differential 
equations that are linearizations of nonlinear equations modeling the Josephson effect in superconductivity. 
We consider the modified Bessel functions $I_j(x)$ of the first kind, which are Laurent series coefficients of the analytic 
function family $e^{\frac x2(z+\frac 1z)}$. For every $l\geq1$ we study  the family parametrized by $k, n\in\zz^l$, 
$k_1>\dots>k_l$, $n_1>\dots>n_l$ of $(l\times l)$-matrix functions formed 
by the modified Bessel functions of the first kind $a_{ij}(x)=I_{k_j-n_i}(x)$, $i,j=1,\dots,l$. 
We show that their determinants  $f_{k,n}(x)$ are positive for every 
$l\geq1$, $k,n\in\zz^l$ as above and $x>0$. 
The above determinants are closely related to a sequence (indexed by $l$) of 
families of double confluent Heun equations, which are linear second order 
differential equations with two irregular singularities, at zero and at infinity. V.M.Buchstaber and S.I.Tertychnyi have constructed their 
holomorphic solutions on $\cc$ for an explicit class of parameter values and  conjectured that they  do not exist for other 
parameter values. 
They have reduced their conjecture to the second conjecture saying that if an appropriate 
 second similar equation has a polynomial solution, then the first one 
has no entire solution. They have proved the latter statement under the additional assumption (third conjecture) 
that  $f_{k,n}(x)\neq0$ for $k=(l,\dots,1)$, 
$n=(l-1,\dots,0)$ and every $x>0$. Our more general result implies 
all the above conjectures, together with their corollary for  
the overdamped model of the Josephson junction in superconductivity: the  description of adjacency points of phase-lock areas as 
solutions of explicit analytic equations.
\end{abstract}

\section{Introduction}

We consider the well-known problem on entire solutions of double confluent Heun equations. Our results are directed to applications 
to nonlinear equations modeling the Josephson effect in superconductivity.

\subsection{Main result}

%{\bf To do.}
%
%
%4)!!! Do we have cite Lea Jaccoud El Jaick Heun equations convergence Coulomb Using Coulomb functions to construct solutions.
%Po-moiemu, statia o drugom, i ssylatsia nie obiazatelno. 
%
%%

Let $Y(\zz^l)$ denote the space of the so-called {\it two-sided Young diagrams} of order $l$: 
$$Y(\zz^l)=\{ k=(k_1,\dots,k_l) \ | \ k_1>\dots>k_l\}\subset\zz^l.$$
This notion is motivated by the fact that the $k\in Y(\zz^l)$ with $k_i>0$ are the classical Young diagrams. 
To every two-sided infinite number sequence $(a_j)_{j=-\infty}^{+\infty}$ and every $l\geq1$ 
we associate the sequence of matrices $A_{k,n}$ numerated by pairs of two-sided Young diagrams $k$ and $n$: 
\begin{equation}A_{k,n}=(a_{ij}), \ a_{ij}=a_{k_j-n_i}; \  f_{k,n}:=\det A_{k,n}.\label{det}\end{equation}

%begin-CCC
\begin{remark} The matrices $A_{k,n}$ with $k=(k_1,k_1-1,\dots,k_1-l+1)$ and $n=(n_1,n_1-1,\dots,n_1-l+1)$ are the classical 
 Hankel matrices \cite[pp. 301--302, 495]{gant} 
 written with inverse order of columns. The theory of Hankel matrices has important 
 applications to the theory of functions, see \cite[chapter XVI]{gant}. The matrices $A_{k,n}$ corresponding to general two-sided 
 Young diagrams may be considered as a natural generalization of Hankel matrices. The results of the paper provide the context on 
 the crossing of dynamical systems, complex differential equations and physics where the latter matrices   naturally arise.  
 \end{remark}
%end-CCC

\begin{remark} For every fixed two-sided Young diagram $n$ 
the determinants $f_{k,n}$ with variable $k\in Y(\zz^l)$  form an infinite  sequence of projective 
Pl\"ucker coordinates corresponding to the subspace generated by the vector 
$a=(a_j)$ and its shifts by $n_1,\dots,n_l$ in an infinite-dimensional space. 
\end{remark}

The main result of the paper is the next theorem, which concerns the above determinants $f_{k,n}(x)$ 
constructed from the sequence of modified Bessel functions of the first kind 
$a_j=I_j(x)$. Recall that $I_j(x)$  are Laurent series coefficients for the family of analytic functions
$$g_x(z)=e^{\frac x2(z+\frac1z)}=\sum_{j=-\infty}^{+\infty}I_j(x)z^j.$$
Equivalently, they are defined by the integral formulas  
$$I_j(x)=\frac1{\pi}\int_0^{\pi}e^{x\cos\phi}\cos(j\phi) d\phi, \ j\in\zz.$$

\begin{theorem} \label{th} The determinant $f_{k,n}(x)$ in (\ref{det}) with $a_j=I_j(x)$ is positive 
for every two-sided Young diagrams $k$ and $n$ and every $x>0$.
\end{theorem}

\begin{remark} Recall that a rectangular $l\times m$-matrix  is called {\it strictly totally positive} (nonnegative) \cite{abgryn, gantkr, gantkr2, pincus}, 
if all its minors of all the dimensions  are positive (nonnegative). Many results on characterization and properties of strictly totally 
positive matrices and their relations to other domains of mathematics, e.g., dynamical systems, mathematical physics, combinatorics, geometry and topology, are given in loc. cit. and in \cite{post, fomzel} (see also references in all these papers and books). A.Postnikov's paper \cite{post} 
deals with the 
matrices $l\times m$, $m\geq l$ or rank $l$ satisfying a weaker condition of nonnegativity of just higher rank minors. One of its 
main results provides an explicit combinatorial cell decomposition 
of the corresponding subset in the Grassmanian $G(l,m)$, called the {\it totally nonnegative Grassmanian.}  The cells are coded by combinatorial types of appropriate planar networks. 
K.Talaska \cite{talaska} obtained further development and generalization of  Postnikov's result. In particular, for a given point 
of the totally nonnegative Grassmanian the results of \cite{talaska} allow to decide what is its ambient cell and what are its affine coordinates in the cell. S.Fomin and A.Zelevinsky \cite{fomzel} studied a more general notion of total 
positivity (nonnegativity) for elements of a semisimple complex Lie group with a given double Bruhat cell decomposition. 
 They have proved that the totally positive parts of the double Bruhat cells are 
bijectively parametrized by the product of the positive quadrant $\rr_+^m$ and the positive subgroup of the maximal torus. 
\end{remark}

Theorem \ref{th} provides an explicit one-dimensional family (given by classical special functions) 
of  strictly totally positive matrices $A_{k,n}$ with $l$ rows 
and {\it infinite} number of columns. We hope that appearance of  strictly totally positive matrices in the context of the 
present paper would open a new direction of their applications.

%Theorem \ref{th} solves positively conjecture 3 from \cite[p. 342]{bt1}. 

Let us describe the relation of construction (\ref{det})  to Schur polynomials. Consider the generating function associated to a sequence $a_j$: 
$$M(a;w)=\sum_{j=-\infty}^{+\infty}a_jw^j, \ w\in\cc.$$
Recall that for every classical Young diagram $n\in Y(\zz^l)$ one denotes 
$$\Delta_n(z_1,\dots,z_l)=\left|\begin{matrix} & z_1^{n_1} & z_2^{n_1} &\dots & z_l^{n_1}\\
& z_1^{n_2} & z_2^{n_2} &\dots & z_l^{n_2}\\
& \dots & \dots &\dots &\dots\\
& z_1^{n_l} & z_2^{n_l} &\dots & z_l^{n_l}\end{matrix}\right|.$$
One denotes
$$M_n(a;z)=\Delta_n(z)\prod_{i=1}^lM(a;z_i), \text{ where } z=(z_1,\dots, z_l).$$

\begin{lemma} One has the following formula 
\begin{equation}M_n(a;z)=\sum_{k\in\zz^l}f_{k,n}z^k.\label{gener}\end{equation}
\end{lemma}

\begin{proof} The Laurent coefficient at $z^k$ of the function $M_n(a;z)$ 
equals the sum $(-1)^{\sigma}a_{k_1-n_{\sigma(1)}}\dots a_{k_l-n_{\sigma(l)}}$ 
taken over all the permutations $\sigma\in S_l$. The latter sum is obviously equal to the determinant $f_{k,n}$, see (\ref{det}).
\end{proof}

\begin{remark} \label{resign}  The determinants $f_{k,n}$ are well-defined  for every $k,n\in\zz^l$. One has 
$f_{k,n}=0$ if either $k_i=k_j$, or $n_i=n_j$ for some $i\neq j$. If a tuple $\wt k\in\zz^l$ is obtained from another tuple 
$k$ by a permutation $\sigma$, 
then $f_{\wt k, n}=(-1)^{\sign(\sigma)}f_{k,n}$, where $\sign(\sigma)$ is the parity of the permutation $\sigma$. Analogous statement 
holds for the other parameter $n$. The function $M_n(a;z)$ is obviously anti-symmetric in $z=(z_1,\dots,z_l)$. This together with 
Theorem \ref{th}  implies that for every given two-sided Young diagram $n$ 
and $a_j=I_j(x)$ its Laurent coefficient with multi-index $k$ is positive, whenever the order of the components $k_j$ differs from 
the decreasing one by an even permutation. 
\end{remark}
Set
$$\delta=(l-1,l-2,\dots,0)\in Y(\zz^l),$$
$$s_{\la}(z)=\frac{\Delta_{\delta+\la}(z)}{\Delta_{\delta}(z)}, \ z=(z_1,\dots,z_l), \ \la\in\zz^l.$$
Recall that if $\la$ is a classical Young diagram, then by definition, $s_{\la}$ is the  {\it Schur polynomial}  associated to $\la$, 
see \cite[p. 40]{mcd}. If $\la\in\zz^l$ is not 
a Young diagram, then we will call $s_{\la}$ a {\it Schur rational function}.
%, cf. loc. cit. 
\begin{corollary} For every $\la\in \zz^l$ one has 
$$M_{\delta+\la}(a;z)=s_{\la}(z) M_{\delta}(a;z).$$
\end{corollary}
 \begin{corollary} 
 $$M_n(a;z)=\sum_{k\in Y(\zz^l)}f_{k,n}\Delta_k(z)=\Delta_{\delta}(z)\sum_{k\in Y(\zz^l)}f_{k,n}s_{k-\delta}(z).$$
 % $$\frac{M_n}{M_{\delta}}(a;z)=s_{n-\delta}(z)=\frac{\Delta_{\delta}(z)}{M_{\delta}(a;z)}\sum_{k\in Y(\zz^l)}f_{k,n}s_{k-\delta}(z).$$
 \end{corollary}
 
 %begin-CCC
 The main application of our results concerns the family 
 \begin{equation}\frac{d\phi}{dt}=-\sin \phi + B + A \cos\omega t, \ A,\omega>0, \ B\geq0.\label{josbeg}\end{equation}
  of  {\it nonlinear} equations, which arises in several models in physics, mechanics and geometry. For example, it describes 
the overdamped model of the 
Josephson junction (RSJ - model)  in superconductivity (our main motivation), see \cite{josephson, stewart, mcc, bar, schmidt}; 
it arises in  planimeters, see  \cite{Foote, foott}. 
Here $\omega$ is a fixed constant, and $(B,A)$ are the parameters. Set 
$$\tau=\omega t, \ l=\frac B\omega, \ \mu=\frac A{2\omega}.$$
The variable change $t\mapsto \tau$ transforms (\ref{josbeg}) to a 
non-autonomous ordinary differential equation on the two-torus $\mathbb T^2=S^1\times S^1$ with coordinates 
$(\phi,\tau)\in\rr^2\slash2\pi\zz^2$: 
\begin{equation} \dot \phi=\frac{d\phi}{d\tau}=-\frac{\sin \phi}{\omega} + l + 2\mu \cos \tau.\label{jostor}\end{equation}
The graphs of its solutions are the orbits of the vector field 
\begin{equation}\begin{cases} & \dot\phi=-\frac{\sin \phi}{\omega} + l + 2\mu \cos \tau\\
& \dot \tau=1\end{cases}\label{josvec}\end{equation}
on $\mathbb T^2$. The {\it rotation number} of its flow, see \cite[p. 104]{arn},  is a function $\rho(B,A)$ of parameters. 
%(Normalization convention: the rotation number of a usual circle rotation equals the rotation angle divided by $2\pi$.) 
%The $B$-axis will be called the {\it abscissa,} and the $A$-axis will be called the {\it ordinate.}

The  {\it phase-lock areas} are the level subsets of the rotation number in the $(B,A)$-plane 
with non-empty interior. They have been studied 
by V.M.Buchstaber, O.V.Karpov, S.I.Tertychnyi et al, see \cite{bt1, 4} and references therein and in Section 3 below.  
Each phase-lock area is an infinite chain of adjacent domains separated by {\it adjacency points}. V.M.Buchstaber and S.I.Tertychnyi 
have described coordinates of a wide class of adjacency points \cite{bt1} and conjectured that this is the complete list of adjacencies. 
This was done via reduction of the family of non-linear equations to two 
families (\ref{heun1}), (\ref{heun2}) of second order linear differential equations of double confluent Heun type. 
They have shown that adjacencies correspond exactly to those parameter values, for which the linear equations have  non-trivial holomorphic 
solutions at 0. They formulated conjectures on the existence of their holomorphic solutions, which would imply the above conjecture 
on the adjacencies. It appears that the first linear equation cannot have polynomial solutions, while the second one can. 
The main conjecture, which implies the others, says that if 
the second linear equation has a polynomial solution, then the first linear equation does not have holomorphic solution at zero. Buchstaber 
and Tertychnyi have reduced it to their other conjecture saying that the determinants $f_{\delta +(1,\dots,1),\delta}(x)$ are non-zero 
for every $x>0$ \cite[conjecture 3, p.342]{bt1}. They have proved their positivity  for $l\leq3$ and arbitrary $x>0$ 
and also for all $l$ and small $x>0$ dependently on $l$ \cite[appendixes 1, 2]{bt2}. 

Theorem \ref{th} is a more general result, which implies all the above conjectures.

Theorem \ref{th} will be proved in Section 2. Its  proof is based on 
the following observation. The derivative of each modified Bessel function $I_j$ is the half-sum of its neighbors $I_{j\pm1}$.  
This implies that the derivative of each determinant  $f_{k,n}$ is a linear combination with positive coefficients 
of other determinants $f_{k',n}$ with $k'$  obtained  from $k$ by adding $\pm1$ to some component. The surprising idea is to write 
the latter formula for derivative not just for a single determinant $f_{k,n}$,  but for all the determinants $f_{k,n}$  with a fixed $n$ 
and all $k\in Y(\zz^l)$ together. 
This yields a linear  ordinary differential equation on the infinite-dimensional vector function $(f_{k,n}(x))|_{k\in Y(\zz^l)}$ with 
the right-hand side being equal to the discrete Laplacian plus the multiplication by $2l$.  The latter right-hand side represents a 
 linear bounded vector field on the space  $l_2$ of infinite sequences $(f_k)_{k\in Y(\zz^l)}$ for which the positive quadrant 
 $\{ f_k\geq0 \ | \ k\in Y(\zz^l)\}$ is an invariant subset. It is shown that the infinite vector function $(f_{k,n}(x))_{k\in Y(\zz^l)}$ 
 is $l_2$-valued.  The values at 0 of all the determinants being non-negative, the initial condition 
 belongs to the positive quadrant. This implies that the above vector of  determinants $f_{k,n}(x)$  lies in the positive quadrant  
 for all $x>0$, and hence, each determinant is non-negative for $x>0$. 
 Its strict positivity is deduced from the same differential equation, which implies that all its derivatives are non-negative and some of them 
does not vanish at zero.

Section 3 presents applications of Theorem \ref{th} to double-confluent Heun equations and  nonlinear equations 
(\ref{josbeg}), (\ref{jostor}) related to the Josephson effect.  
It contains an introduction to the subject, an overview of previous results and proof of the above-mentioned conjectures of V.M.Buchstaber and S.I.Tertychnyi on Heun equations and adjacencies of phase-lock areas. 
%end-CCC

\section{Proof of Theorem \ref{th}}
In the proof of Theorem \ref{th} we use the following classical properties of the modified Bessel functions $I_j$ of the first kind, see 
\cite[section 3.7]{watson}. 
\begin{equation} I_j=I_{-j}; \label{sym}\end{equation}
\begin{equation} I_j|_{x>0}>0; \ I_j(0)=0 \text{ for } j\neq0; \ I_0(0)>0;\label{ioo}\end{equation}
\begin{equation} I_0'=I_1; \ I_j'=\frac12(I_{j-1}+I_{j+1});\label{differ}\end{equation}
\begin{equation} I_j(2y)=\sum_{s=0}^{\infty}\frac{y^{j+2s}}{s!(j+s)!} \text{ for every } j\geq0.\label{series}\end{equation}
The next two propositions and corollary together imply that for every fixed $n\in Y(\zz^l)$ the vector function $(f_{k,n}(x))_{k\in Y(\zz^l)}$ 
  is a solution of a bounded linear ordinary differential equation in the Hilbert space $l_2$ 
of infinite sequences $(f_k)_{k\in Y(\zz^l)}$: a phase curve of a bounded linear vector field. 
 We show that the positive quadrant $\{ f_k\geq0 \ | \ k\in Y(\zz^l)\}\subset l_2$ is invariant under the positive flow of the latter field. This implies 
 that $f_{k,n}(x)\geq0$ for all $x\geq0$, and then we easily deduce that the latter inequality is strict for $x>0$. This will prove Theorem \ref{th}. 

Let us recall how the discrete Laplacian $\Delta_{discr}$ acts on the space of functions 
$f=f(k)$ in $k\in\zz^l$. For every $j=1,\dots,l$ let $T_j$ denote the 
corresponding shift operator: 
$$(T_jf)(k)=f(k_1,\dots,k_{j-1},k_j-1, k_{j+1},\dots,k_l).$$
Then 
\begin{equation} \dd=\sum_{j=1}^l(T_j+T_j^{-1}-2).\label{lapl1}\end{equation}
Thus, one has 
$$(\Delta_{discr}f)(p)=\sum_{s=1}^l(f(p_1,\dots,p_{s-1},p_s-1,p_{s+1},\dots,p_l)$$
\begin{equation}+f(p_1,\dots,p_{s-1},p_s+1,p_{s+1},\dots,p_l))-2lf(p).\label{lapl2}\end{equation}
\begin{remark} \label{restrict} We will deal with the class of sequences $f(k)$ for which $f(k)=0$ whenever $k_i=k_j$ for some $i\neq j$. It includes   
$f(k)=f_{k,n}(x)$ with fixed $n\in\zz^l$ and $x\in\mathbb R$, see Remark \ref{resign}. In this case 
the discrete Laplacian is well-defined by the above formulas (\ref{lapl1}), (\ref{lapl2}) 
on the restrictions of the latter sequences  $f(k)$ to the two-sided Young diagrams $k\in Y(\zz^l)$. 
\end{remark}

\begin{proposition} For every $l\geq1$ and $n\in Y(\zz^l)$ the vector function $(f(x,k)=f_{k,n}(x))_{k\in Y(\zz^l)}$ satisfies the following linear differential equation:
\begin{equation} \frac{\partial f}{\partial x}=\Delta_{discr}f+2lf.
\label{diffs}\end{equation}
%\frac12\sum_{j=1}^l(f_{(k_1,\dots,k_{j-1}, k_j-1, k_{j+1},\dots,k_l)}(x)+
%f_{(k_1,\dots,k_{j-1}, k_j+1,k_{j+1}\dots,k_l)}(x)).
\end{proposition}

Equation (\ref{diffs}) follows immediately from definition, equation (\ref{differ}) and Remarks \ref{resign}, \ref{restrict}. 

\begin{remark} \label{new} For every $k\in Y(\zz^l)$ the $k$-th component of the right-hand side in (\ref{diffs}) is a linear combination with 
strictly positive coefficients of 
the components $f(x,k')$ with $k'\in Y(\zz^l)$ obtained from $k=(k_1,\dots,k_l)$ by adding $\pm1$ to some $k_i$. This follows from 
(\ref{lapl2}), (\ref{diffs}). 
\end{remark}
%
%\begin{remark} The right-hand side in equation (\ref{diffs}) can be expressed via the discrete Laplacian $\Delta_k$  in the variable $k$. 
%Its value on a  function $a(k)$, $k\in\zz^l$, equals $(\Delta_ka)(j)=a(j-1)+a(j+1)-2a(j)$. Thus, 
% system (\ref{diffs}) is equivalent to the following differential equation on the function $f_k(x)=f(x,k)$ in two variables: 
% $$\frac{\partial f}{partial x}=\frac12\Delta_kf+2f.$$
% \end{remark}

\begin{proposition} \label{term} For every constant $R>1$ and every $j\geq R^2$ one has 
\begin{equation}|I_j(x)|<\frac{R^{j}}{j!} \text{ for every }  0\leq x\leq R.\label{ineqj}\end{equation}
\end{proposition}
\begin{remark} \label{decr} The sequence $\frac{R^j}{j!}$ is bounded,  and it decreases in $j\geq R$. 
\end{remark}
\begin{proof} {\bf of Proposition \ref{term}.} Fix an arbitrary $j\geq R^2$. Let us estimate the terms of the series (\ref{series}). For every $s\geq0$ 
and $y\in[0,\frac R2]$ one has 
$$\frac{y^{j+2s}}{s!(j+s)!}\leq\frac{R^{j+2s}}{2^{j+2s}(j+s)!}\leq\frac{R^j}{j!}\frac1{2^{j+2s}}\left(\frac{R^2}j\right)^{s}\leq\frac1{2^{j+2s}}\frac{R^j}{j!}.$$
This together with (\ref{series}) implies (\ref{ineqj}). 
\end{proof}

\begin{corollary} \label{cl2} For every $l\geq1$, $n\in Y(\zz^l)$ and $x\geq0$ one has $(f_{k,n}(x))_{k\in Y(\zz^l)}\in l_2$. Moreover, 
there exists a function $C(R)=C_n(R)>0$ in $R>1$ such that 
\begin{equation}\sum_{k\in Y(\zz^l)}|f_{k,n}(x)|^2<C(R) \ \text{ for every } 0\leq x\leq R.\label{hilb}\end{equation}
\end{corollary}
\begin{proof} Fix an $n\in Y(\zz^l)$ and an $R>1+|n|_{\max}$, $|n|_{\max}=\max_i|n_i|$. Set 
$$|k|_{n,\max}=|k|_{\max}-|n|_{\max}; \ K=K(n,R)=\{ k\in Y(\zz^l) \ | \ |k|_{n,\max}\geq R^2\}.$$ 
It suffices to prove uniform boundedness in $x\in[0,R]$ of  sum (\ref{hilb}) taken through all $k\in K$, since the complement 
$Y(\zz^l)\setminus K$ is finite. Set 
$$M=\max_{j\in\zz, \ 0\leq x\leq R}I_j(x).$$
The number $M$ is finite, by (\ref{ineqj}) and Remark \ref{decr}. For every $k\in K$ one has 
\begin{equation} |f_{k,n}(x)|<
\frac{R^{|k|_{n,\max}}}{(|k|_{n,\max})!}l!M^{l-1} \text{ for every }  0\leq x\leq R.\label{ineqf}\end{equation} 
Indeed, if $k\in K$, then some column of the matrix $A_{k,n}$ consists of functions $I_j$ with 
$j\geq |k|_{n,\max}\geq R^2$, 
which satisfy inequality (\ref{ineqj}), by Proposition \ref{term} and (\ref{sym}). For the latter $j$  the 
right-hand sides of inequality (\ref{ineqj}) are no greater than 
$\frac{R^{|k|_{n,\max}}}{(|k|_{n,\max})!}$, whenever $x\in[0,R]$, by Remark \ref{decr}. The other matrix elements are no greater that $M$ on 
$[0,R]$. Therefore, the module $|f_{k,n}(x)|$ of its determinant  defined as sum of $l!$ products of functions $I_j$ satisfies inequality 
(\ref{ineqf}). This implies that the sum in (\ref{hilb}) through $k\in K$ is no greater than 
$$C(R)=l!M^{l-1}\sum_{k\in K}\frac{R^{|k|_{n,\max}}}{(|k|_{n,\max})!}<+\infty.$$ 
The corollary is proved.
 \end{proof}
 
\begin{definition} Let $\Omega$ be the closure of an open convex subset in a Banach space. For every $x\in\partial\Omega$ consider the 
union of all the rays issued from $x$ that intersect $\Omega$ in at least two distinct points (including $x$). The closure of the latter union of rays 
is a convex cone, which will be here referred to, as the {\it generating cone\footnote{The authors believe that this definition and the next proposition 
are well-known to specialists, but they have not found them in literature.}} $K(x)$.
\end{definition}

\begin{proposition} \label{inv} Let $H$ be a Banach space, $\Omega\subset H$ be as above. Let $v$ be a $C^1$ vector field on a neighborhood of 
the set $\Omega$ in $H$ such that $v(x)\in K(x)$ for every $x\in\partial\Omega$. 
Then the set $\Omega$ is invariant under the flow of the field $v$: each positive semitrajectory starting at $\Omega$ is 
contained in $\Omega$. 
\end{proposition}

\begin{proof} Fix a point $O\in Int(\Omega)$. Consider the ``anti-Euler'' vector field $w$ on $H$: its vector at a point 
$x\in H$ is the vector $xO$ directed to $O$. Consider the family of vector fields $u_\var=v+\var w$. For every $\var>0$ 
and $x\in\partial\Omega$ one has $u_\var(x)\in Int(K(x))$. In other words, the field $u_\var$ with $\var>0$ is directed 
strictly inside the domain $\Omega$, and its trajectories obviously stay in $\Omega$. Hence, the latter statement holds for $\var=0$, by the 
theorem on the existence and uniqueness of solutions of bounded ordinary differential equation and continuity in parameter \cite{ch}. 
This proves the proposition.
\end{proof}

Now let us prove Theorem \ref{th}. Fix an $n\in Y(\zz^l)$. 
The right-hand side of differential equation (\ref{diffs}) is a bounded linear vector field on the Hilbert space $l_2$ 
of sequences $(f_k)_{k\in Y(\zz^l)}$. We will denote the latter vector field by $v$. Let 
$\Omega\subset l_2$ denote the ``positive quadrant'' defined by the inequalities $f_k\geq0$. For every point $x\in\partial\Omega$ the 
vector $v(x)$ lies in its generating cone $K(x)$: the components of the field $v$ are non-negative on $\Omega$, by Remark \ref{new}.  
The vector function $(f_{k,n}(x))_{k\in Y(\zz^l)}$ in $x\geq0$ is an $l_2$-valued  solution of the corresponding differential equation, 
by Corollary \ref{cl2}. One has $(f_{k,n}(0))_{k\in Y(\zz^l)}\in\Omega$:
\begin{equation}f_{k,n}(0)=0 \text{ whenever } k\neq n; \ f_{n,n}(0)=I_0^l(0)>0,\label{foo}\end{equation}
which follows from (\ref{ioo}). This together with Proposition \ref{inv} implies that 
\begin{equation} f_{k,n}(x)\geq0 \text{ for every } k\in Y(\zz^l) \text{ and } x\geq0.\label{nonneg}\end{equation}
 Now let us prove that the inequality is strict for all $k\in Y(\zz^l)$ and $x>0$. Indeed, let 
$f_{p,n}(x_0)=0$ for some $p=(p_1,\dots,p_l)\in Y(\zz^l)$ and $x_0>0$. All the derivatives of the function 
$f_{p,n}$ are non-negative, by (\ref{diffs}), Remark \ref{new} and (\ref{nonneg}). Therefore,  $f_{p,n}\equiv0$ 
on the segment $[0,x_0]$. This together with (\ref{diffs}), Remark \ref{new}  
and (\ref{nonneg}) implies that $f_{p',n}\equiv0$ on $[0,x_0]$ for every $p'$ obtained from $p$ by adding $\pm1$ to some component. 
We then get by induction that $f_{n,n}(0)=0$, -- a contradiction to (\ref{foo}). The proof of Theorem \ref{th} is complete.

\section{Applications to double confluent Heun equations and Josephson effect: entire solutions and adjacencies}

Here we prove the conjectures of V.M.Buchstaber and S.I.Tertychnyi from \cite{bt1} mentioned in the introduction. 
They concern the family of nonlinear equations (\ref{josbeg}): 
 \begin{equation}\dot \phi=\frac{d\phi}{dt}=-\sin \phi + B + A \cos\omega t, \ A,\omega>0, \ B\geq0.\label{jos}\end{equation}
We fix an arbitrary $\omega>0$ and consider family (\ref{jos}) depending on two variable parameters $(B,A)$. The variable change $\tau=\omega t$ 
transforms (\ref{jos}) to the differential equation (\ref{jostor}) on the two-torus $\mathbb T^2=S^1\times S^1$ with coordinates 
$(\phi,\tau)\in\rr^2\slash2\pi\zz^2$. Its solutions are tangent to the vector field 
\begin{equation}\begin{cases} & \dot\phi=-\frac{\sin \phi}{\omega} + l + 2\mu \cos \tau\\
& \dot \tau=1\end{cases}, \ \ l=\frac B{\omega}, \ \mu=\frac A{2\omega}\label{josvect}\end{equation}
on the torus. The {\it rotation number} of the equation (\ref{jos}) is, by definition, the rotation number 
of the flow of the field (\ref{josvect}), see \cite[p. 104]{arn}. It is a function $\rho(B,A)$ of parameters. 
(Normalization convention: the rotation number of a usual circle rotation equals the rotation angle divided by $2\pi$.) 
The $B$-axis will be called the {\it abscissa,} and the $A$-axis will be called the {\it ordinate.}

\begin{definition} (cf. \cite[definition 1.1]{4}) The {\it $l$-th phase-lock area} is the level set 
$\{B,A) \ | \ \rho(B,A)=l\}\subset\rr^2$, provided it has a non-empty interior. 
\end{definition}

\begin{remark}{\bf: phase-lock areas and Arnold tongues.} The behavior of phase-lock areas for small $A$ demonstrates the Arnold 
tongues effect \cite[p. 110]{arn}. The phase-lock areas are called ``Arnold tongues'' in \cite[definition 1.1]{4}. 
\end{remark}

Recall that the rotation number of system (\ref{jos}) has the physical meaning of the mean voltage over a long
time interval. The segments in which the phase-lock areas intersect horizontal lines correspond to
the Shapiro steps on the voltage-current characteristic.

It has been shown earlier that

- the phase-lock areas  exist only for  integer values of the rotation number (a ``quantization effect'' observed in \cite{buch2} and later also proved 
in  \cite{IRF, LSh2009}); 

- the boundary of each phase-lock area $\{\rho = l\}$ consists of two analytic curves, which are the graphs of two
functions $B=g_{l,\pm}(A)$ (see \cite{buch1}; this fact was later explained by A.V.Klimenko via symmetry, see \cite{RK}); 

-  the latter functions have Bessel asymptotics (observed and proved on physical level in ~\cite{shap}, see also \cite[chaptrer 5]{lich},
 \cite[section 11.1]{bar}, ~\cite{buch2006}; proved mathematically in ~\cite{RK}).
 
 - each phase-lock area is an infinite chain of bounded domains going to infinity in the vertical direction, 
 each two subsequent domains are separated by one point, the separation points lying outside the horizontal $B$-axis are called the 
 {\it adjacency points} (or briefly {\it adjacencies)};
 
 - for every $l\in\zz$ the $l$-th phase-lock area is symmetric to the $-l$-th one with respect to the $A$-axis (symmetry of equation (\ref{jos}); the set 
 of adjacencies of all the phase-lock areas is also symmetric). 

 There is a conjecture saying that for every $l\in\zz$ the adjacencies of the $l$-th phase-lock area lie on the vertical line $B=l\omega$, see Fig.1. This conjecture is an open problem that was supported numerically  in 
 \cite{4}\footnote{The results of paper \cite{4} concern a slightly 
 different family of differential equations equivalent to (\ref{jos}), namely, $\frac{dx}{d\tau}=\nu\sin x + a + s\sin \tau$. 
 It is obtained from (\ref{jos}) by coordinate and 
 parameter change $\tau=\frac{\pi}2-\omega t$, $x=-\phi$, $\nu=\frac1{\omega}$, $a=\frac B{\omega}$, $s=\frac A{\omega}$}.  
  It was rigorously shown  in loc. cit. that for every adjacency $(B,A)$ one has $l=\frac B\omega\in\zz$, $l\equiv\rho(B,A)(mod 2)$ 
  and $|l|\leq|\rho(B,A)|$. 
  
  Theorem \ref{adj}, one of the main results of the section stated and proved in Subsection 3.2 describes the adjacencies lying on 
  a given line $B=l\omega$, $l\in\zz$, $l\geq0$, as solutions of an explicit analytic equation. To prove the conjecture, one has to show 
  that their rotation numbers are equal to $l$.

\begin{figure}[ht]
  \begin{center}
   \epsfig{file=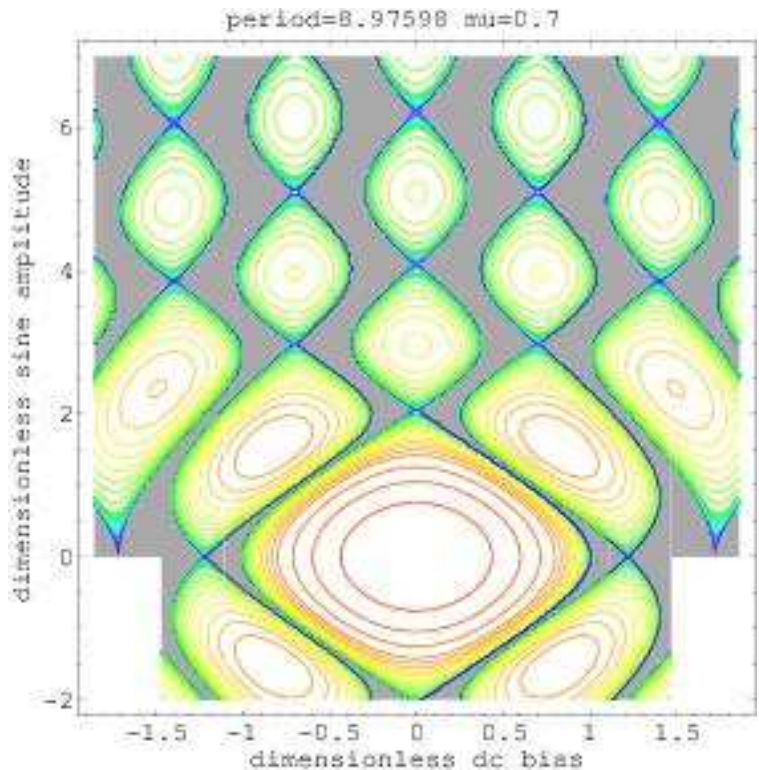}
    \caption{Phase-lock areas and their adjacencies for $\omega = 0.7$. The abscissa is $B$, the ordinate is $A$. 
     Figure taken from \cite[p. 331]{bt1}.}
    %Digram of zone of phase lock
    % }
    \label{fig:1}
  \end{center}
\end{figure} 
% \begin{figure}[htb]
%    \begin{center}
%	\subfigure[$\nu=2.5$]{\includegraphics[scale=0.8]{tongues-gl-2.pdf}}
%	\subfigure[$\nu=1$]{\includegraphics[scale=0.8]{tongues-gl.pdf}}
%	
%	\subfigure[$\nu=0.5$]{\includegraphics[scale=0.8]{tongues-gl-3.pdf}}
%    \caption{Arnold tongues with numbers 0, 1, 2, 3, 4. The adjacencies have integer abscissas equal to the number of the tongue.}
%    \label{fig:1}
%  \end{center}
%  \end{figure}
% Theorem \ref{adj} stated and proved in Subsection 3.2 gives the  description  of the adjacencies as solutions of explicit analytic equations. 
 The proof of  Theorem \ref{adj} is based on the results of V.M.Buchstaber and S.I.Tertychnyi \cite{bt1}, Theorem \ref{th} and the next theorem  relating 
 family (\ref{jos}) to families of double confluent Heun type equations. To state it, let us recall the following constructions and notations from 
 \cite{bt1}. Set
\begin{equation} l=\frac B\omega,  \ \mu=\frac A{2\omega}, \ \lambda=\left(\frac1{2\omega}\right)^2-\mu^2,\label{param}\end{equation}
The adjacencies correspond to $l\in\zz$, and it suffices to describe only those with $l\geq0$, by symmetry. Thus, without loss of generality, everywhere below we consider that $l\in\zz$, $l\geq0$. 

To family of equations (\ref{jos}), V.M.Buchstaber and S.I.Tertychnyi have  associated in \cite{bt1} the  family of
 second order differential equations
\begin{equation} z^2E''+((l+1)z+\mu(1-z^2))E'+(\la-\mu(l+1)z)E=0,\label{heun1}\end{equation}
which is equivalent to the family of double confluent Heun equations
\begin{equation} z^2v''+((l+1)z+\mu(z^2+1))v'+(\la+\mu^2)v=0, \ \ \ v=e^{-\mu z}E,\label{heun}\end{equation}
cf \cite[equations (2), (23)]{bt1}. 

\begin{theorem} \label{holojos} For every $\omega>0$, $l\in\zz$, $l\geq0$  a pair $(B,A)$ with $A\neq0$, $B=l\omega$ 
 is an adjacency for family of equations (\ref{jos}), if and only if 
the corresponding equation (\ref{heun1}) with $\la$, $\mu$ as in (\ref{param}) has a nontrivial holomorphic solution at 0. 
\end{theorem}

\begin{remark} A solution of equation (\ref{heun1})  is holomorphic at 0, if and only if it is an entire function: 
holomorphic on $\cc$. An entire solution is uniquely defined up to multiplicative constant. See  
\cite[lemma 3, statement 4]{bt2}. 
\end{remark}

Theorem \ref{holojos}  was implicitly stated in \cite[p. 332, paragraph 2]{bt1}. We give its proof  in Subsection 3.2 for 
completeness of presentation.

In Subsection 3.1 we describe completely those parameter values for which equation (\ref{heun1}) has a nontrivial entire solution: 
these parameter values are solutions of equation (\ref{xio}).  We then deduce the description of the adjacencies in Subsection 3.2.

\subsection{Entire solutions of double confluent Heun equations}
For every $\la,\mu\in\rr$, $\mu\geq0$ and $l\in\zz$, $l\geq0$ set 
\begin{equation}R_m=\prod_{j=m}^{\infty}M_j, \text{ where } M_j=\left(\begin{matrix} &1+\frac{\la}{j(j-l)} & \frac{\mu^2}{j(j-l)}\\
& 1 & 0\end{matrix}\right), \ m=l+1,l+2,\dots,\label{rm}\end{equation}
where the multipliers with bigger indices are placed to the right from those with smaller indices. The well-definedness of the 
infinite products $R_m$ was proved in \cite[lemma 15]{bt1}. Set
\begin{equation}\xi_l(\la,\mu)=(\la \ \  \mu^2) R_{l+1}\left(\begin{matrix} & 1\\ & 0\end{matrix}\right).\label{xi}\end{equation}

The main result of the present subsection is the following theorem. 

\begin{theorem} \label{equiv} For every $l\geq0$ equation (\ref{heun1}) has a nontrivial entire solution, if and only if the corresponding 
parameters $(\la,\mu)$ satisfy the equation 
\begin{equation}\xi_l(\la,\mu)=0.\label{xio}\end{equation}
\end{theorem} 

Theorem \ref{equiv} answers positively conjecture 2 from \cite[p. 332]{bt1}. 

\begin{remark} Equation (\ref{heun1}) in an entire function $E(z)$ can be translated as a system of linear three-term recurrent 
relations on its Taylor coefficients, see \cite[p.338, formula (34)]{bt1}. The above matrices $M_j$ and function $\xi_l(\la,\mu)$ introduced in \cite[p.337]{bt1} originate from studying the latter 
recurrence relations. The function $\xi_l(\la,\mu)$ is analytic, and its Taylor series converges on the whole complex 
plane $\cc^2$. This follows from the convergence of the infinite matrix product $R_m$, see \cite[lemma 15]{bt1}. (Formally, the 
proof  of convergence in loc. cit. is given for $\la,\mu\in\rr$, but it remains valid for complex values without changes.)  
\end{remark}

The proof of Theorem \ref{equiv} is based on the idea from \cite{bt1} to consider 
simultaneously the equation

\begin{equation} z^2E''+((-l+1)z+\mu(1-z^2))E'+(\mu(l-1)z+\la)E=0,\label{heun2}\end{equation}
which is obtained from equation (\ref{heun1}) by changing $l$ to $-l$. 

%Theorem \ref{equiv} will be deduced below from the  following four theorems.  

\begin{theorem} \label{lo} \cite[theorem 7]{bt1} The statement of Theorem \ref{equiv} holds for $l=0$.
\end{theorem}

\begin{theorem} \label{suffice} \cite[theorem 2]{bt1} For every $l\in\zz$, $l\geq0$, and every $(\la,\mu)$ satisfying equation (\ref{xio}) the corresponding equation (\ref{heun1}) has a nontrivial entire solution.
\end{theorem}

\begin{theorem} \label{pol-all} \cite[theorem 8, p. 353]{bt1} 
 Let for a given $l\in\mathbb N$ and some $\la$, $\mu$ the corresponding 
differential equation (\ref{heun2}) have no polynomial 
solutions. Then equation (\ref{xio}) is also a necessary condition for the existence of a non-trivial entire solution of equation (\ref{heun1}). 
\end{theorem}

\begin{theorem} \label{polsol} If for a given $l\geq1$ 
equation (\ref{heun2}) has a polynomial solution, then the corresponding equation (\ref{heun1}) has 
no nontrivial entire solution.
\end{theorem}

Theorem \ref{polsol} solves positively conjecture 1 in  \cite[p. 332]{bt1}. 

\begin{proof} Theorem 2 from \cite{bt2} says that the statement of Theorem \ref{polsol} holds under the additional condition that 
 the determinant $\Delta(x)=f_{k,n}(x)$, $k=(l,\dots,1)$, $n=(l-1,\dots,0)$ from (\ref{det}) with $a_j=I_j(x)$ 
is non-zero for every $x>0$. But the latter inequality follows immediately from Theorem \ref{th}. This proves Theorem \ref{polsol}.
\end{proof}

\begin{remark} The proof of theorem 2 from \cite{bt2}, which relates the existence  of polynomial solution of equation (\ref{heun2}) 
to the above determinant $\Delta(x)$, is done as follows. 
If equation (\ref{heun1}) has an entire solution, then each its solution is a holomorphic function on $\cc^*=\cc\setminus\{0\}$ 
whose Laurent series contains no monomial $z^{s}$, $-l\leq s\leq-1$, see \cite[lemma 3, part (6)]{bt2}. Suppose, by contradiction, that 
the  corresponding equation (\ref{heun2}) has a polynomial solution $\hat E$. Then $deg\hat E=l-1$, by \cite[remark 3, p.973]{bt0}. 
The transformation 
$$g(z)\mapsto e^{\mu(z+z^{-1})}g(-z^{-1})$$
sends solutions of equation (\ref{heun2}) to solutions of equation (\ref{heun1}) \cite[lemma 3, part (5)]{bt2}. Therefore, the function 
\begin{equation}E(z)=e^{\mu(z+z^{-1})}\hat E(-z^{-1})\label{romb}\end{equation}
is a solution of equation (\ref{heun1}). The vector of its Laurent coefficients  at powers $z^{s}$, $s=-l,\dots,-1$ 
should be equal to zero. On the other hand, it is obtained 
from the vector of coefficients of the polynomial $\hat E$  (written in appropriate order with appropriate signs) by multiplication by 
the matrix $A_{k,n}$  from (\ref{det}) of modified Bessel functions $a_j=I_j(2\mu)$ with the above $k$ and $n$:  
$\Delta(2\mu)=det A_{k,n}$. This follows immediately from formula (\ref{romb}). Therefore, if $\Delta(2\mu)\neq0$, then 
the above Laurent coefficient vector is nonzero, -- a contradiction.
\end{remark}

\begin{proof} {\bf of Theorem \ref{equiv}.} Theorem \ref{equiv} follows from Theorems  \ref{lo}--\ref{polsol}.
\end{proof}

\subsection{Josephson effect: adjacencies of phase-lock areas}

\begin{theorem} \label{adj}  For every given $\omega>0$ and $l\in\zz$, $l\geq0$, set $B=l\omega$, a pair  $(B,A)\in\rr^2$ with $A\neq0$ 
is an adjacency of the corresponding family of equations (\ref{jos}), if and only if the corresponding parameters $\la$, $\mu$ given by 
(\ref{param}) satisfy equation (\ref{xio})
\end{theorem}

Theorem \ref{adj} follows from Theorems \ref{holojos} and \ref{equiv}. 
%Let us give a proof of Theorem \ref{holojos}.

\begin{proof} {\bf of Theorem \ref{holojos}.} Set  
$$\Phi=e^{i\phi}, \ z=e^{i\tau}=e^{i\omega t}.$$
The complexified equation (\ref{jos}) is equivalent to the Riccati equation 
$$\frac{d\Phi}{dz}=z^{-2}((lz+\mu(z^2+1))\Phi-\frac z{2i\omega}(\Phi^2-1)).$$
The latter is the projectivization of the following linear equation in vector function $(u,v)$ with $\Phi=\frac v{u}$:
\begin{equation}\begin{cases} & v'=\frac1{2i\omega z}u\\
& u'=z^{-2}(-(lz+\mu(1+z^2))u+\frac z{2i\omega}v)\end{cases}\label{tty}\end{equation}
This reduction to a system of linear equations was earlier obtained in slightly different terms in \cite{bkt1, Foote, bt1, IRF}. 
It is easy to check that a function $v(z)$ is the component of a solution of system (\ref{tty}), if and only if it satisfies double confluent 
Heun equation (\ref{heun}), or equivalently, the function $E(z)=e^{\mu z}v(z)$ satisfies equation (\ref{heun1}). System (\ref{tty})  
has singularities only at zero and at infinity; both are irregular ones.

Let us suppose that given $l\geq0$ and $(\la,\mu)$ correspond to an adjacency. Then the corresponding linear system (\ref{tty}) (and hence, 
equation (\ref{heun1})) has trivial monodromy operator along a positive circuit around the origin. This follows from the proof of \cite[lemma 3.3]{4}: 
it was shown in loc. cit. that the monodromy matrix should be equal to $\diag(1,e^{2\pi ia})$ with $a\in\zz$, and hence, to the identity. This implies that 
each solution of equation (\ref{heun1}) is holomorphic on $\cc^*=\cc\setminus\{ 0\}.$ Exactly one non-trivial solution (up to multiplicative 
constant) should be holomorphic at zero. This follows from the fact that the germ at 0 of system (\ref{tty}) should be analytically equivalent to its 
diagonal formal normal form \cite[lemma 3.3]{4} and from \cite[proposition 2.9]{4}. The first part of Theorem \ref{holojos} is proved. 

Now let us prove the converse. Let equation (\ref{heun1}) have a nontrivial solution holomorphic at 0. Then all its solutions are holomorphic in 
$\cc^*$, by \cite[theorem 3]{bt1}. This implies that equation (\ref{heun1}) (and hence, system (\ref{tty})) has trivial monodromy. This together 
with \cite[proposition 3.2]{4} implies that the parameters under consideration correspond to an adjacency. Theorem \ref{holojos} is proved. 
The proof of Theorem \ref{adj} is complete.
\end{proof}

\section{Acknowledgement}

We are grateful to the referees for very stimulating reports and  useful remarks.  We are grateful to 
P.Grinevich for helpful discussions and informing us about related results on total positivity 
\cite{abgryn, fomzel, post, talaska}.


\begin{thebibliography}{}

\bibitem{abgryn} Abenda, S.; Grinevich, P. Rational degeneration of M-curves, totally positive Grassmannians and KP-solitons. Preprint   https://arxiv.org/abs/1506.00563

\bibitem{arn} Arnold, V. I. {\it Geometrical Methods in the Theory of Ordinary Differential Equations,} Second edition. 
Grundlehren der Mathematischen Wissenschaften [Fundamental Principles of Mathematical Sciences], 250. Springer-Verlag, 
New York, 1988.

\bibitem{bar} Barone, A.; Paterno, G. {\it Physics and Applications of the Josephson Effect,} John Wiley and
Sons, New York--Chichester--Brisbane--Toronto--Singapore, 1982.

\bibitem{bkt1} Buchstaber, V.M.;  Karpov, O.V.; Tertychnyi, S.I. {\it On properties of the differential
equation describing the dynamics of an overdamped Josephson junction,}  Russian Math. Surveys, \textbf{59:2} (2004), 377--378.

\bibitem{buch2006} Buchstaber, V.M.; Karpov, O.V.; and Tertychnyi, S.I. {\it Peculiarities of dynamics of a Josephson
junction shifted by a sinusoidal SHF current} (in Russian).  Radiotekhnika i Elektronika, \textbf{51:6} (2006), 757--762.

\bibitem{buch2} Buchstaber, V.M.; Karpov, O.V.;  Tertychnyi, S.I. {\it The rotation number quantization effect}, Theoret and Math. Phys., \textbf{162} 
 (2010), No. 2,  211--221.

\bibitem{buch1} Buchstaber, V.M.; Karpov, O.V.; Tertychnyi, S.I. {\it The system on torus modeling the dynamics of Josephson junction,} 
Russ. Math. Surveys, \textbf{67}  (2012), 178--180.
%\textbf{67}:1 (403) (2012)181--182?,  

\bibitem{bt0} Buchstaber, V.M.; Tertychnyi, S.I. {\it Explicit solution family for the equation of the resistively shunted Josephson junction model.,} Theoret. and Math. Phys., \textbf{176} (2013), No.2, 965--986. 

\bibitem{bt1} Buchstaber, V.M.; Tertychnyi, S.I. {\it Holomorphic solutions of the double confluent Heun equation associated with the RSJ model of the Josephson junction,} Theoret. and Math. Phys., \textbf{182:3} (2015), 329--355.

\bibitem{bt2} Buchstaber, V.M.; Tertychnyi, S.I. {\it A remarkable sequence of Bessel matrices,} Mathematical Notes, \textbf{98} (2015),
 No. 5,  714--724.
 
\bibitem{ch} Choquet-Bruhat, Y., \ de Witt-Morette, C., \ Dillard-Bleick, M. {\it Analysis, 
Manifolds and Physics,}  North-Holland, 1977.

\bibitem{fomzel} Fomin, S.; Zelevinsky, A. {\it Double Bruhat cells and total
positivity,}  Journal of the American Mathematical Society \textbf{12} 
(1999), No. 2, 335--380. 

\bibitem{Foote} Foote, R.L., {\it Geometry of the Prytz Planimeter,} Reports on Math. Phys. \textbf{42:1/2} (1998), 249--271.

\bibitem{foott} Foote, R.L.; Levi, M.; Tabachnikov, S. {\it Tractrices, bicycle tire tracks, hatchet planimeters, and a 100-year-old conjecture,} 
Amer. Math. Monthly, \textbf{103} (2013), 199--216.

\bibitem{gant} Gantmacher, F.R., {\it The theory of matrices} (in Russian),  Second edition. Moscow, Nauka, 1966. 

\bibitem{gantkr} Gantmacher, F. R.; Krein, M. G. 
{\it Oscillation matrices and small oscillations of mechanical systems.} (in Russian),  Moscow--Leningrad, 1941.

\bibitem{gantkr2} Gantmacher, F. R.; Krein, M. G. 
{\it Oscillation matrices and kernels and small oscillations of mechanical systems.} (in Russian),  2d ed. Gosudarstv. Isdat. Tehn.-Teor. 
Lit., Moscow-Leningrad, 1950.

\bibitem{4} Glutsyuk, A.A.; Kleptsyn, V.A.; Filimonov, D.A.; Schurov, I.V. {\it On the adjacency quantization in an equation modeling the 
Josephson effect,} Funct. Analysis and Appl., \textbf{48} (2014), No.4, 272--285.

\bibitem{LSh2009} Ilyashenko, Yu.S. {\it Lectures of the summer school ``Dynamical systems'',} Poprad, Slovak Republic, 2009.

\bibitem{IRF} Ilyashenko, Yu.S.; Filimonov, D.A.; Ryzhov, D.A.  {\it Phase-lock effect for equations modeling resistively shunted
Josephson junctions and for their perturbations,}  Funct. Analysis and its Appl. \textbf{45} (2011), No. 3, 192--203.

\bibitem{josephson} Josephson, B.D., {\it Possible new effects in superconductive tunnelling,} Phys. Lett., \textbf{1} (1962), No. 7, 
251--253. 

\bibitem{RK} Klimenko, A.V; Romaskevich, O.L. {\it Asymptotic properties of Arnold tongues
and Josephson effect,}  Mosc. Math. J., \textbf{14:2} (2014), 367--384.

\bibitem{lich} Likharev, K.K.; Ulrikh, B.T. {\it Systems with Josephson junctions: Basic Theory,} Izdat.
MGU, Moscow, 1978.

\bibitem{mcd} Macdonald, I. {\it Symmetric functions and Hall polynomials,} Second Edition. Clarendon Press,  Oxford, 1995.  

\bibitem{mcc} McCumber, D.E. {\it Effect of ac impedance on dc voltage-current characteristics of superconductor weak-link junctions,}  J. Appl. Phys., \textbf{39} (1968), No.7,   3113--3118. 

\bibitem{pincus} Pinkus, A., {\it Totally positive matrices,}  Cambridge Tracts in Mathematics, \textbf{181}. Cambridge
University Press, Cambridge, 2010. xii+182 pp. ISBN: 978-0-521-19408-2.

\bibitem{post} Postnikov A., {\it Total positivity, Grassmannians, and networks.} Preprint  https://arxiv.org/abs/math/0609764

\bibitem{schmidt} Schmidt, V.V., {\it Introduction to physics of superconductors} (in Russian), MCCME, Moscow, 2000. 

\bibitem{shap} Shapiro, S.; Janus, A.; Holly, S. {\it Effect of microwaves on Josephson currents in superconducting
tunneling,}  Rev. Mod. Phys., \textbf{36} (1964), 223--225.

\bibitem{stewart} Stewart, W.C., {\it Current-voltage characteristics of Josephson junctions.} Appl. Phys. Lett., \textbf{12} (1968), No. 8, 277--280. 

\bibitem{talaska} Talaska, K., {\it Combinatorial formulas for} \_I{\it-coordinates in a totally
nonnegative Grassmannian,} J. Combin. Theory, Series A \textbf{118} (2011), 58--66. 


\bibitem{watson} Watson, G.N., {\it A treatise on the theory of Bessel functions (2nd. ed.).}  Vol. 1,   Cambridge University Press, 1966. 
\end{thebibliography}
\end{document}